\documentclass{amsart}
\usepackage{graphicx}
\usepackage{amssymb}
\usepackage{enumitem}

\newtheorem{thm}{Theorem}[section]

\newtheorem{lem}[thm]{Lemma}

\theoremstyle{definition}

\theoremstyle{remark}
\newtheorem{rem}[thm]{Remark}

\begin{document}

\title[non $C^1$ solutions]{Non $C^1$ solutions to the special Lagrangian equation}
\author{Connor Mooney}
\address{Department of Mathematics, UC Irvine}
\email{\tt mooneycr@math.uci.edu}
\author{Ovidiu Savin}
\address{Department of Mathematics, Columbia University}
\email{\tt savin@math.columbia.edu}

\begin{abstract}
We construct viscosity solutions to the special Lagrangian equation that are Lipschitz but not $C^1$, and have non-minimal gradient graphs.
\end{abstract}
\maketitle

\section{Introduction}
For a symmetric $n \times n$ matrix $M$ with eigenvalues $\{\lambda_i\}_{i = 1}^n$, we let
$$F(M) = \sum_{i = 1}^n \arctan(\lambda_i).$$
The special Lagrangian equation for a function $u$ on a domain in $\mathbb{R}^n$ is
\begin{equation}\label{sLag}
F(D^2u) = c \in \left(-n \frac \pi 2,\,n \frac \pi 2\right).
\end{equation}
Here $c$ is a constant. Equation (\ref{sLag}), introduced in the seminal work \cite{HL1}, is the potential equation for area-minimizing Lagrangian graphs of dimension $n$ in $\mathbb{R}^{2n}$. 

Classical solvability of the Dirichlet problem for (\ref{sLag}) in a ball with smooth boundary data was established for $|c|$ large in \cite{CNS} (see also \cite{CPW}, \cite{BMS}, \cite{Lu}, and see \cite{BW} for classical solvability of the second boundary value problem). The existence of viscosity solutions to (\ref{sLag}) with continuous boundary data and $c$ arbitrary was established in \cite{HL2} (see also \cite{B}, \cite{DDT}, \cite{CP}, and \cite{HL3}).

The regularity of solutions to (\ref{sLag}) is a delicate issue. It is known that viscosity solutions are real analytic when $|c| \geq (n-2)\frac{\pi}{2}$ (\cite{WY1}, see also \cite{WaY1}, \cite{WaY2}, \cite{WaY3}, \cite{Li}, \cite{Z}, \cite{T2}, \cite{U}, \cite{CW}), or when $u$ is convex (\cite{CSY}, see also \cite{CWY}, \cite{BC}, \cite{BCGJ}). In these cases, (\ref{sLag}) is a concave equation \cite{Y3}, so by the Evans-Krylov theorem (\cite{E}, \cite{K}) it suffices to obtain interior $C^2$ estimates (see also \cite{Y1}). When $|c| < (n-2)\frac{\pi}{2}$ the equation is not concave, and there are examples of viscosity solutions to (\ref{sLag}) which are $C^1$ but not $C^2$ (\cite{NV2}, \cite{WY2}). However, the gradient graphs of these examples are analytic and area-minimizing as geometric objects. It remained open whether all viscosity solutions to (\ref{sLag}) are $C^1$, and whether they have minimal gradient graphs (see e.g. the conjecture at the end of the introduction in \cite{NV2}). In this paper we answer these questions in the negative:

\begin{thm}\label{Main}
There exist $c \in [0,\,\pi/2)$, a smooth bounded domain $\Omega \subset \mathbb{R}^3$, an analytic embedded surface $\Gamma \subset \subset \Omega$ with boundary, and a Lipschitz function $u$ on $\Omega$ that is analytic in $\Omega \backslash \Gamma$, such that 
\begin{enumerate}
\item $F(D^2u) = c$ in the viscosity sense in $\Omega$,
\item $\nabla u$ is discontinuous on $\Gamma \backslash \partial \Gamma$, and
\item The graph of $\nabla u$ is of class $C^{1,\,1}$ but not $C^2$. It is the union of two analytic parts, where one of the parts is minimal and the other is not.
\end{enumerate}
\end{thm}

Here we clarify what we mean in the last statement. Our approach is to first construct a $C^{2,\,1}$ solution $w$ to the degenerate Bellman equation
\begin{equation}\label{Bellman}
\text{max}\{F(D^2w) - c^*,\,\det D^2w\} = 0
\end{equation}
which has a compact free boundary between the operators. The function $w$ solves the special Lagrangian equation outside of a small smooth convex set $K$, in which $\det D^2w = 0$ but $F(D^2w)$ is not constant. It is analytic inside $K$ and outside $\overline{K}$, and $C^{2,\,1}$ but not $C^3$ across $\partial K$. Thus, the graph $\{(x,\,\nabla w(x))\}$ of $\nabla w$ consists of two analytic parts that meet in a $C^{1,\,1}$ but not $C^2$ fashion, where one part is minimal (the part where $x \in \overline{K}^c$) and the other is not (where $x \in K$ and $F(D^2w)$ is not constant). To get $u$ we take the Legendre transform of $w$, and we interpret the graph of $\nabla u$ as a rigid motion in $\mathbb{R}^3 \times \mathbb{R}^3$ of the graph of $\nabla w$. This is a natural interpretation in view of the fact that the gradient of the Legendre transform of a function is the inverse of the gradient of that function (see Section \ref{MainExample}).

We also remark that the example $u$ from Theorem \ref{Main} is semi-concave, hence $-u$ (which solves the special Lagrangian equation with right hand side $-c$) is semi-convex, like the examples in \cite{NV2} and \cite{WY2}. However, in contrast with previous examples, the example in Theorem \ref{Main} has non-minimal and non-smooth gradient graph (in the sense we describe above). Thus, unlike convexity, semi-convexity does not imply smoothness of the gradient graph for solutions of (\ref{sLag}). 

Theorem \ref{Main} also says something interesting at the level of $C^1$ estimates for degenerate elliptic PDEs, namely, that solutions that are smooth near the boundary (which guarantees interior gradient bounds by the comparison principle) can have interior gradient discontinuities. This stands in contrast with uniformly elliptic equations, which enjoy interior $C^{1,\,\alpha}$ estimates (\cite{C}, \cite{CC}, \cite{T1}; these are in fact optimal, see \cite{NV3}). It remains open whether viscosity solutions to (\ref{sLag}) necessarily have locally bounded gradient, see e.g. \cite{M} and \cite{BMS} for recent results in this direction.

On the other hand, all the known non-smooth solutions to (\ref{sLag}) are not $C^{1,\,1}$. It remains open whether there exist $C^{1,\,1}$ but non $C^2$ singular solutions (equivalently, whether there exist non-flat graphical special Lagrangian cones). The smallest dimension in which such examples could exist is $n = 5$ (\cite{NV1}). For general uniformly elliptic equations, such examples exist in dimensions $n \geq 5$ (\cite{NTV}).

We will prove Theorem \ref{Main} in the following section. In the last section we will discuss related examples of singular solutions to (\ref{sLag}), that can be viewed as small perturbations of the singular solutions in \cite{NV2} and \cite{WY2}. The examples in the last section have non-minimal gradient graphs, and the singularities appear near the center of a ball. We expect that the examples in the last section are not $C^1$, and that their singularities are modeled locally by examples like the one from Theorem \ref{Main}. In particular, we expect that the degenerate Bellman equation (\ref{Bellman}) with compact free boundaries plays an important role in the formation of Lipschitz singularities in solutions to (\ref{sLag}). Furthermore, the examples in the last section suggest that this mechanism of singularity formation is stable.
 
\section*{Acknowledgments}
C. Mooney was supported by a Sloan Fellowship, a UC Irvine Chancellor's Fellowship, and NSF CAREER Grant DMS-2143668.
O. Savin was supported by the NSF grant DMS-2055617.
The authors are grateful to the referees and to Yu Yuan for helpful comments.

\section{Proof of Theorem \ref{Main}}\label{MainExample}
For $\lambda > 0$ to be chosen shortly, let
$$\Phi(x) = \frac{\lambda x_1^2}{1+x_3} + \frac{\lambda x_2^2}{1-x_3}.$$
The function $\Phi$ is convex and analytic in $\{|x_3| < 1\}$. (Note that the one-homogeneous function of two variables $x_1^2/x_3$ is convex in $\{x_3 > 0\}$, since it has only one nonzero Hessian eigenvalue and has positive second derivative in the $x_1$ direction. The terms in $\Phi$ are this same function up to rigid motions, so the convexity of $\Phi$ follows.) Each term in $\Phi$ is a translation of a one-homogeneous function whose Hessian has rank $1$, so $D^2\Phi$ has rank $2$. It follows that
$$\det D^2\Phi = 0.$$
Note also that the image of $\nabla \Phi$ is contained in the paraboloid
\begin{equation}\label{Sigma}
\Sigma := \left\{y_3 = \frac{1}{4\lambda}(y_2^2 - y_1^2)\right\}.
\end{equation}

\begin{lem}\label{LocalMin}
The analytic function $\Theta(x) = F(D^2\Phi(x))$ has a non-degenerate local minimum at $x = 0$ for $\lambda > 0$ small.
\end{lem}

\begin{proof}
Denote by $f(\lambda):= \arctan \lambda$ and then $F(D^2 u)= \sum f(\lambda_i)$.

We use the expansion of order 4 for $\Phi$ at $0$,
$$ \Phi:= \lambda \left(x_1^2 + x_2^2 + x_3(x_2^2-x_1^2) + x_3^2 (x_1^2 + x_2^2) + O(|x|^5) \right),$$
hence $D^2 \Phi(0)$ is diagonal with eigenvalues $2\lambda$, $2\lambda$ and $0$. 

 We compute $ D \Theta$ and $D^2 \Theta$ at $x=0$ and find
$$ \Theta_k= F_{ij} \,  \Phi_{ijk}=f'(2\lambda)\Phi_{11k} + f'(2\lambda) \Phi_{22k} + f'(0) \Phi_{33k} =0,$$
and
\begin{align*}
\Theta_{kl}& = F_{ij}\,  \Phi_{ijkl} + F_{ij,mn} \Phi _{ijk} \Phi_{mnl} \\
& =  f'(2\lambda)\Phi_{11kl} + f'(2\lambda) \Phi_{22kl} + f'(0) \Phi_{33kl} + O(\lambda^2) \\
&= \Phi_{11kl} +  \Phi_{22kl} + \Phi_{33kl} + O(\lambda^2)\\
& = c_k \lambda \, \delta_k^l + O(\lambda^2),
\end{align*}
with $c_1=c_2=4$ and $c_3=8$. In the computation above the derivatives of $F$ are evaluated at $D^2 \Phi(0)$, and we have used that $$D^3 \Phi = O(\lambda), \quad D^4 \Phi = O(\lambda), \quad f'(2 \lambda) = f'(0) + O(\lambda), \quad f'(0)=1.$$ 

Hence, if $\lambda >0$ is chosen small, then $D^2\Theta(0)$ is positive definite, and the lemma is proved.

\end{proof}

\begin{rem}
Lemma \ref{LocalMin} can also be proven by calculating the eigenvalues of $D^2\Phi$. For $a = \frac{1}{1+x_3}$ and $b = \frac{1}{1-x_3},$ these are $0$ and $\Lambda_{\pm}$, where
$$\frac{1}{2\lambda}\Lambda_{\pm} = \frac{1}{1-x_3^2} + \frac{a^3}{2}x_1^2 + \frac{b^3}{2} x_2^2 \pm \left(\frac{1}{4} (2abx_3 + (b^3x_2^2 - a^3x_1^2))^2 + a^3b^3x_1^2x_2^2 \right)^{1/2}.$$
\end{rem}

Lemma \ref{LocalMin} implies that for $\epsilon > 0$ small, the connected component $K$ of the set $\{\Theta \leq \Theta(0) + \epsilon^2\}$ containing the origin is compact, analytic, uniformly convex, and contained in $B_{C\epsilon}$. Here and below $C$ denotes a large constant, which may change from line to line. As a result, $D^2\Phi$ is within $C\epsilon$ of the diagonal matrix $D^2\Phi(0) = 2\lambda(I - e_3 \otimes e_3)$ in $K$. Later we will use the map
$$\Psi(x) := (\Phi_1(x),\,\Phi_2(x),\,x_3),$$
which is an analytic diffeomorphism in a neighborhood of $0$. Since
$$D^2(\Theta \circ \Psi^{-1})(0) = (D\Psi^{-1}(0))^TD^2\Theta(0)(D\Psi^{-1}(0))$$
is positive, we also have for $\epsilon$ sufficiently small that $\Psi(K),$ the connected component of $\{\Theta \circ \Psi^{-1} \leq \Theta(0) + \epsilon^2\}$ containing the origin, is analytic and uniformly convex.

Now, for $c^* := \Theta(0) + \epsilon^2$, let $v$ be the solution in a small neighborhood of $\partial K$ to
$$F(D^2v) = c^*, \quad v|_{\partial K} = \Phi,\, \quad v_{\nu}|_{\partial K} = \Phi_{\nu}.$$
Here $\nu$ is the outer unit normal to $\partial K$ and we obtain $v$ using Cauchy–Kovalevskaya. Since $\Phi$ and $v$ solve the same equation on $\partial K$ and have the same Cauchy data there, we have
$$D^2v = D^2\Phi \text{ on } \partial K.$$ 
As a result, for $x_0 \in \partial K$, all third derivatives of $\Phi$ and $v$ that involve a differentiation in a direction tangent to $\partial K$ at $x_0$ agree. Since $\Theta_{\nu} > 0$ on $\partial K$ by construction and $F(D^2v)$ is constant, we conclude on $\partial K$ that
$$0 < \partial_{\nu}(\Theta - F(D^2v)) = F_{ij}(\Phi_{ij\nu}-v_{ij\nu}) = F_{\nu\nu}(\Phi_{\nu\nu\nu}-v_{\nu\nu\nu}),$$
which implies that
\begin{equation}\label{Ineq1}
v_{\nu\nu\nu} < \Phi_{\nu\nu\nu} \text{ on } \partial K.
\end{equation}

We let $K_{\mu}$ denote the set of points a distance less than $\mu$ from $K$.
\begin{lem}\label{DetSign}
We have $\det D^2v < 0$ on $K_{\mu} \backslash K$, for $\mu > 0$ small.
\end{lem}
\begin{proof}
Let $G$ denote determinant. Since $G(D^2v) = G(D^2\Phi) = 0$ on $\partial K$, it suffices to show that 
$$\mbox{$\partial_{\nu}(G(D^2v)) \leq 0$ on $\partial K$, and where equality holds, that $\partial_{\nu}^2(G(D^2v)) < 0.$}$$ To that end we fix $x_0 \in \partial K$, and we let $\xi$ denote the eigendirection at $x_0$ corresponding to the $0$ eigenvalue of $D^2 \Phi(x_0)$. We distinguish two cases.

The first case is that $\xi$ is not tangent to $\partial K$. Then at $x_0$ we pick a system of coordinates with $\nu$ being a coordinate direction, and at $x_0$ we compute
\begin{align*}
\partial_{\nu}(G(D^2v)) &= \partial_{\nu}(G(D^2v) - G(D^2\Phi)) \\
&= G_{ij}(v_{ij\nu} - \Phi_{ij \nu}) \\
&= G_{\nu\nu}(v_{\nu\nu\nu} - \Phi_{\nu\nu\nu}),
\end{align*}
using that $v_{ijk}(x_0) = \Phi_{ijk}(x_0)$ unless $i = j = k = \nu$. Since $\nu$ is not perpendicular to $\xi$, we have that $G_{\nu\nu}(D^2(\Phi(x_0))) > 0$, and we obtain the desired (strict) inequality using (\ref{Ineq1}).

The second case is that $\xi$ is tangent to $\partial K$. Choose coordinates at $x_0$ such that both $\xi$ and $\nu$ are coordinate directions. In these coordinates the only nonzero derivative of $G$ is $G_{\xi\xi} > 0$. In particular, $G_{\nu\nu} = 0$, so the previous calculation implies that $\partial_{\nu}(G(D^2v)) = 0$. Combining these observations we have
$$0 = \partial_{\nu}(G(D^2v)) = \partial_{\nu}(G(D^2\Phi)) = G_{\xi\xi}v_{\xi\xi\nu} = G_{\xi\xi}\Phi_{\xi\xi\nu},$$
hence
\begin{equation}\label{Vanishing}
v_{\xi\xi\nu} = \Phi_{\xi\xi\nu} = 0.
\end{equation}
Now we calculate the second normal derivative:
\begin{align*}
\partial_{\nu}^2(G(D^2v)) &= \partial_{\nu}^2(G(D^2v) - G(D^2\Phi)) \\
&= G_{\xi\xi}(v_{\xi\xi\nu\nu} - \Phi_{\xi\xi\nu\nu}) + G_{ij,\,kl}(v_{ij\nu}v_{kl\nu} - \Phi_{ij\nu}\Phi_{kl\nu}) \\
&= I + II.
\end{align*}
Since all third-order derivatives of $v$ and $\Phi$ involving a tangential direction agree, the only possible nonzero terms in $II$ are those with $i=j=\nu$ or $k = l = \nu$. Using (\ref{Vanishing}), we further reduce $II$ to terms involving $G_{\nu\nu,\,kl}$ where $k$ and $l$ are not both $\xi$. Finally, using that $D^2\Phi(x_0)$ vanishes in the $\xi$ column and row, we see that $G_{\nu\nu,\,kl} = 0$ when $(k,\,l) \neq (\xi,\,\xi)$, thus the term $II$ vanishes.

To estimate the term $I$ note that
$$(v_{\nu\nu} - \Phi_{\nu\nu})_{\xi\xi} = \kappa(v_{\nu\nu\nu} - \Phi_{\nu\nu\nu}),$$
where $\kappa > 0$ is the curvature of $\partial K$ in the direction $\xi$. Using (\ref{Ineq1}) we conclude that
$$\partial_{\nu}^2(G(D^2v)) = \kappa G_{\xi\xi}(v_{\nu\nu\nu} - \Phi_{\nu\nu\nu}) < 0,$$
completing the proof.
\end{proof}

Now, we let
$$w = \begin{cases}
\Phi \text{ in } K, \\
v \text{ in } K_{\mu} \backslash K.
\end{cases}$$
Note that $w \in C^{2,\,1}$ and
\begin{equation}\label{HessianApprox}
 D^2w \text{ is within } C\epsilon \text{ of the matrix } 2\lambda(I - e_3 \otimes e_3) \text{ in } K_{\mu}
 \end{equation} 
 (we assume $\mu$ was taken small). We let
$$\Gamma := \nabla w(K),$$
and we note that $\Gamma$ is the piece of the paraboloid $\Sigma$ (see (\ref{Sigma})) that lies over the projection of $\Psi(K)$ to the horizontal plane.

\begin{lem}\label{Injective}
For $\mu' > 0$ small, the map $\nabla w$ is one-to-one on $K_{\mu'} \backslash K$ and maps $K_{\mu'} \backslash K$ diffeomorphically to a neighborhood of $\Gamma$.
\end{lem}
\begin{proof}
Let $y = H(x) := (w_1(x),\,w_2(x),\,x_3)$. Similar calculations to those in Proposition 3.1 from \cite{WY2} imply that $\det DH > 0$ and that $H$ is distance-expanding, up to a factor depending on $\lambda$. Both facts follow quickly from (\ref{HessianApprox}), which implies that $DH$ is within $C\epsilon$ of the diagonal matrix with entries $2\lambda,\, 2\lambda,\, 1$ in $K_{\mu}$. (To verify distance-expanding one can e.g. combine the preceding observation with the fundamental theorem of calculus, which implies that $H(z)-H(x) = [\int_0^1 DH(tz + (1-t)x)\,dt]\cdot [z-x])$. In particular, $H$ is a global diffeomorphism of $K_{\mu}$. As noted above, the set $H(K) = \Psi(K)$ is an analytic uniformly convex set. Thus, for $\mu' > 0$ sufficiently small, $H(K_{\mu'})$ is contained in a convex neighborhood $D \subset H(K_{\mu})$ of $H(K)$. We will show that $\nabla w$ is injective in $K_{\mu'} \backslash K$.

Because $H$ is a diffeomorphism, it suffices to check that $T := \nabla w \circ H^{-1}$ is injective on $D \backslash H(K)$. We have
$$T(y) = (y_1,\, y_2,\, w_3(H^{-1}(y))).$$
Since $D$ is convex, every vertical line intersects it in a connected segment, so it is enough to show that $\partial_3T^3 = \partial_{y_3}(w_3(H^{-1}(y))) \leq 0$, with strict inequality when $y \in D \backslash H(K)$. This follows directly from the identity
\begin{equation}\label{Monotone}
\partial_{y_3}(w_3(H^{-1}(y))) = \det DT(y) = \det D^2w(H^{-1}(y))\det DH^{-1}(y)
\end{equation}
and Lemma \ref{DetSign}.

It just remains to show that $\nabla w$ maps $K_{\mu'}$ into a neighborhood of $\Gamma$. Equivalently, $T$ maps $H(K_{\mu'})$ into a neighborhood of $\Gamma$. Using the monotonicity $\partial_3T^3 < 0$ away from $H(K)$ we see that the image of $T$ contains a small vertical segment through every point in $\Gamma$, and the result follows from the continuity of $T$.
\end{proof}

For a $C^2$ function $w$, we define its Legendre transform $w^*$ on the image of the gradient of $w$ by the formula
\begin{equation}\label{LT}
w^*(\nabla w) = x \cdot \nabla w - w(x),
\end{equation}
with $w^*$ being possibly multivalued. Although the Legendre transform is typically used for convex functions, this definition enjoys some of the same important properties. More precisely, if $\det D^2 w (x_0)\ne 0$, then in a neighborhood of $x_0$ the Legendre transform $w^*$ is single-valued, $\nabla w^*$ is the inverse of $\nabla w$, and $D^2 w^* = (D^2 w)^{-1}$, hence
$$F(D^2 w^*) + F(D^2 w) = (n- 2 l) \frac \pi 2,$$
where $l$ denotes the number of negative eigenvalues of $D^2w$. Geometrically, taking the Legendre transform corresponds to making a rigid motion of the gradient graph, which can be seen using the gradient-inverting property.

Using Lemma \ref{Injective} we conclude that there exists a neighborhood of $\Gamma$ on which the Legendre transform 
$u = w^*$
of $w$ is single-valued. Away from $\Gamma$, the function $u$ is analytic and has two positive Hessian eigenvalues and one negative Hessian eigenvalue, thus it solves
\begin{equation}\label{finaleq}
F(D^2u) = \frac{\pi}{2} - c^* := c
\end{equation}
classically away from $\Gamma$. We also calculate away from $\Gamma$ that
$$u_{33} = \frac{1}{\det(D^2w)}\text{cof}(D^2w)_{33} < 0,$$
and $u_{33}$ tends to $-\infty$ on $\Gamma$. On $\Gamma \backslash \partial \Gamma$, the function $u$ has a ``downward" Lipschitz singularity. Indeed, from the identity
$$u_3(y_1,\,y_2,\,w_3(H^{-1}(y))) = u_3(\nabla w(H^{-1}(y))) = y_3$$
we infer that $u_3^+$ and $u_3^-$, the limits of $u_3$ from above and below along vertical segments through $\Gamma \backslash \partial \Gamma$, satisfy that 
$$(u_3^+ - u_3^-)(y_1,\,y_2,\,w_3(H^{-1}(y))) = -L(y_1,\,y_2) < 0,$$
where $L(y_1,\,y_2)$ is the length of the intersection between $\Psi(K)$ and the vertical line through $(y_1,\,y_2,\,0)$.

We conclude from this discussion that $u$ is concave along vertical lines, and on $\Gamma$ it cannot be touched from below by any $C^2$ function. As a consequence, $u$ is a viscosity super-solution to (\ref{finaleq}).
Note also that $F(D^2w) \leq c^*$. It follows that $w_k := w - x_3^2/k$ satisfies $F(D^2w_k) < c^*$ for all $k > 0$. We note that $D^2w_k$ has two positive eigenvalues and one negative eigenvalue, and by similar considerations to those above, $w_k^*$ is single-valued and solves
$$F(D^2w_k^*) = \pi/2 - F(D^2w_k) > c$$
classically in a neighborhood of $\Gamma$. We claim that $w_k^*$ converge uniformly to $u$ in a neighborhood of $\Gamma$, which implies that $u$ is also a viscosity sub-solution to (\ref{finaleq}) and completes the construction. To prove the claim, note that by the definition of Legendre transform (\ref{LT}),
$$w_k^*(\nabla w(x) - 2k^{-1}x_3e_3) - w^*(\nabla w(x)) = -k^{-1}x_3^2,$$
and use that $w_k^*$ have uniformly bounded gradient (their gradients lie in $K_{\mu}$).

\begin{rem}
By combining the proofs of Lemmas \ref{DetSign} and \ref{Injective} one can show that $u$ is $C^{1,\,1/2}$ up to $\Gamma$ from each side at points on $\Gamma \backslash \partial \Gamma$, and $u$ is $C^{1,\,1/5}$ on $\partial \Gamma$. Indeed, for $z_0 \in \Gamma \backslash \partial \Gamma$, let $x_0$ be either pre-image under $\nabla w$ of $z_0$ in $\partial K$. Lemma \ref{DetSign} shows that $\partial_{\nu}(G(D^2w)) < 0$ at $x_0$. Using this in (\ref{Monotone}) one can conclude that $|\nabla w(x) - \nabla w(x_0)| \geq C^{-1}|x - x_0|^2$ for $x \in B_r(x_0) \cap (K_{\mu'} \backslash K)$ and $r$ small, giving $C^{1/2}$ regularity of $\nabla u$ on each side of $\Gamma$ at $z_0$. Likewise, if $z_0 \in \partial \Gamma$, then one has $\partial_{\nu}^2(G(D^2w)) < 0$ at $x_0$. Using the uniform convexity of $\partial K$ one concludes in a similar way using (\ref{Monotone}) that $|\nabla w(x) - \nabla w(x_0)| \geq C^{-1}|x-x_0|^5$ for $x \in K_{\mu'} \backslash K$, corresponding to $C^{1/5}$ regularity of $\nabla u$ at $z_0$.
\end{rem}

\section{Related Examples}
The examples in \cite{NV2} and \cite{WY2} are obtained by starting with an analytic solution to the special Lagrangian equation with singular Hessian at the origin and injective gradient. The gradient graph can then be rotated so it has a ``vertical" tangent direction at the origin, and the new potential (the Legendre transform of the original one) is $C^1$ but not $C^2$. Rotating the gradient graphs a tiny bit further gives rise to a potential that is multi-valued in a small neighborhood of the origin. By solving the Dirichlet problem for (\ref{sLag}) with boundary data given by that of the multi-valued potential, one obtains solutions that cannot have minimal gradient graph. Here and below, by minimal we mean a mass-minimizing integral varifold. In this section we outline a proof, and we discuss the relationship between these examples and the one from the previous section. The idea of working in rotated coordinate systems has been used to prove regularity and Liouville-type theorems in many contexts, see e.g. \cite{Y2}, \cite{CY}, and \cite{CSY}.

\vspace{3mm}   

{\bf Step 1:} Calculations in Section $2$ of \cite{WY2} show that there is a solution $w$ to the special Lagrangian equation with $c = \pi/2$ in $B_{2\kappa} \subset \mathbb{R}^3$ such that
\begin{align*}
w &= \frac{1}{2}\left(x_1^2 + x_2^2\right) + x_3(x_1^2-x_2^2) \\
&+ \frac{1}{12}x_3^2\left(18 x_1^2 + 18 x_2^2 - x_3^2\right) - \frac{1}{8}(x_1^2 + x_2^2)^2 + O(|x|^5).
\end{align*}
It is shown that the two largest eigenvalues of $D^2w$ are close to $1$, and that the smallest eigenvalue $\lambda_3$ of $D^2w$ is analytic near the origin and satisfies
\begin{equation}\label{QuadSep}
\lambda_3 = -|x|^2 + O(|x|^3).
\end{equation}

\begin{rem}
The expansion of $w$ follows from the form of the polynomial $P$ on pg. 1161 of \cite{WY2}. We took $m = 2$, which determines the coefficients $a_0 = -1,\, a_1 = 6,\, a_2 = -3/2$ (see the bottom of pg. 1162). Taking $\nu = 1/12$ and using the formulae for $H$ and $\lambda_3$ on pages 1162, 1163 gives the conclusions above. 
\end{rem}

Let $\epsilon > 0$ be small and let $\tan\theta = \epsilon$. Let $(x,\,y)$ be coordinates of $\mathbb{R}^6$ with $x$ and $y$ in $\mathbb{R}^3$. Rotating the gradient graph of $w$ by an angle $\theta$, that is, representing the graph in new coordinates
$$\tilde{x} = \cos\theta x - \sin\theta y, \quad \tilde{y} = \sin\theta x + \cos\theta y,$$ 
we get a new potential $\tilde{w}$ which satisfies
$$\nabla \tilde{w}(\cos\theta x - \sin\theta \nabla w(x)) = \sin\theta x + \cos\theta \nabla w(x),$$
$$D^2\tilde{w}(\cos\theta x - \sin\theta \nabla w(x)) = (I-\epsilon D^2w(x))^{-1}(\epsilon I + D^2w(x)),$$
$$F(D^2\tilde{w}) = \frac{\pi}{2} + 3\theta \text{ in } B_{\kappa}.$$
Letting $\tilde{\lambda}_3$ be the smallest eigenvalue of $D^2\tilde{w}$, we conclude using (\ref{QuadSep}) that
$$\tilde{\lambda}_3(\cos\theta x - \sin\theta \nabla w(x)) = \epsilon - (1+\epsilon^2)|x|^2 + O(|x|^3),$$
hence $\tilde{\lambda}_3(0) = \epsilon,\, \nabla \tilde{\lambda}_3(0) = 0$, and 
$$D^2\tilde{\lambda}_3(0) = -\frac{2(1+\epsilon^2)}{\cos^{2}\theta}\left((1-\epsilon)^{-2}(e_1 \otimes e_1 + e_2 \otimes e_2) + e_3 \otimes e_3\right).$$
For $\epsilon$ small, the connected component $Z$ of the set $\{\tilde{\lambda}_3 > 0\}$ containing the origin is thus an analytic uniformly convex set contained in $B_{C\sqrt{\epsilon}}$. (Here and below, $C$ and $c$ will denote constants independent of $\epsilon$ that may change as we refer to them.) Furthermore, for 
$$\Psi(x) := (\tilde{w}_1,\, \tilde{w}_2,\, x_3),$$
the same is true for the set $\Psi(Z)$, that is, the connected component of $\{\tilde{\lambda}_3 \circ \Psi^{-1} > 0\}$ containing the origin.

\vspace{3mm}

{\bf Step 2:} The Legendre transform $\tilde{w}^*$ of $\tilde{w}$ is defined in a ball $B_{2d}$, and for $\epsilon$ small it is analytic and single-valued in $B_{2d} \backslash B_{d/8}$. Let $u$ be the viscosity solution to the Dirichlet problem
$$F(D^2u) = -3\theta \text{ in } B_{3d/2}, \quad u|_{\partial B_{3d/2}} = \tilde{w}^*,$$
the existence of which was established in \cite{HL2}. We claim that for $\epsilon$ small $u$ is smooth in $B_d \backslash B_{d/2}$, and furthermore that
$$\|u - \tilde{w}^*\|_{C^2(B_{d} \backslash B_{d/2})} < C\epsilon^{2}.$$
To show this, it suffices to establish that 
\begin{equation}\label{C0bound}
\|u-\tilde{w}^*\|_{L^{\infty}(B_{3d/2} \backslash B_{d/4})} < C\epsilon^{2}.
\end{equation} 
Indeed, the small perturbations theorem in \cite{Sa} then implies that $u$ is smooth and bounded in $C^{2,\,\alpha}$ in $B_{5d/4} \backslash B_{3d/8}$. Applying the Schauder interior estimates (\cite{GT}) to the difference of $u$ and $\tilde{w}^*$ (which solves a linear equation with coefficients bounded in $C^{\alpha}$ by the preceding observation) implies the desired $C^2$ estimate.

We show (\ref{C0bound}) using barriers. First, using the convexity of $\Psi(Z)$ and arguments similar to those in Lemma \ref{Injective}, one can show that the preimages under $\nabla \tilde{w}$ of vertical lines are nearly vertical curves that have connected intersection with $Z$. As one follows one of these curves upwards, $\tilde{w}_3$ decreases when the curve lies outside of $Z$ and increases when it is inside of $Z$. This means that $\tilde{w}^*$ is multivalued in a simple way: the graph of $\tilde{w}^*$ along a vertical line is either single valued and concave, or it consists of two crossing concave pieces that lie below and are connected by a convex piece (see Figure \ref{Fig1}). Since $\tilde{w}^*$ solves the dual equation $F(D^2\tilde{w}^*) = -3\theta$ where it is concave in the vertical direction, we conclude that the function $\text{min}(\tilde{w}^*)$ given by the minimum of the possible values of $\tilde{w}^*$ is a super-solution to the equation solved by $u$. In particular, $u \leq \text{min}(\tilde{w}^*)$.

\begin{figure}
 \begin{center}
    \includegraphics[scale=0.4, trim={0mm 100mm 0mm 20mm}, clip]{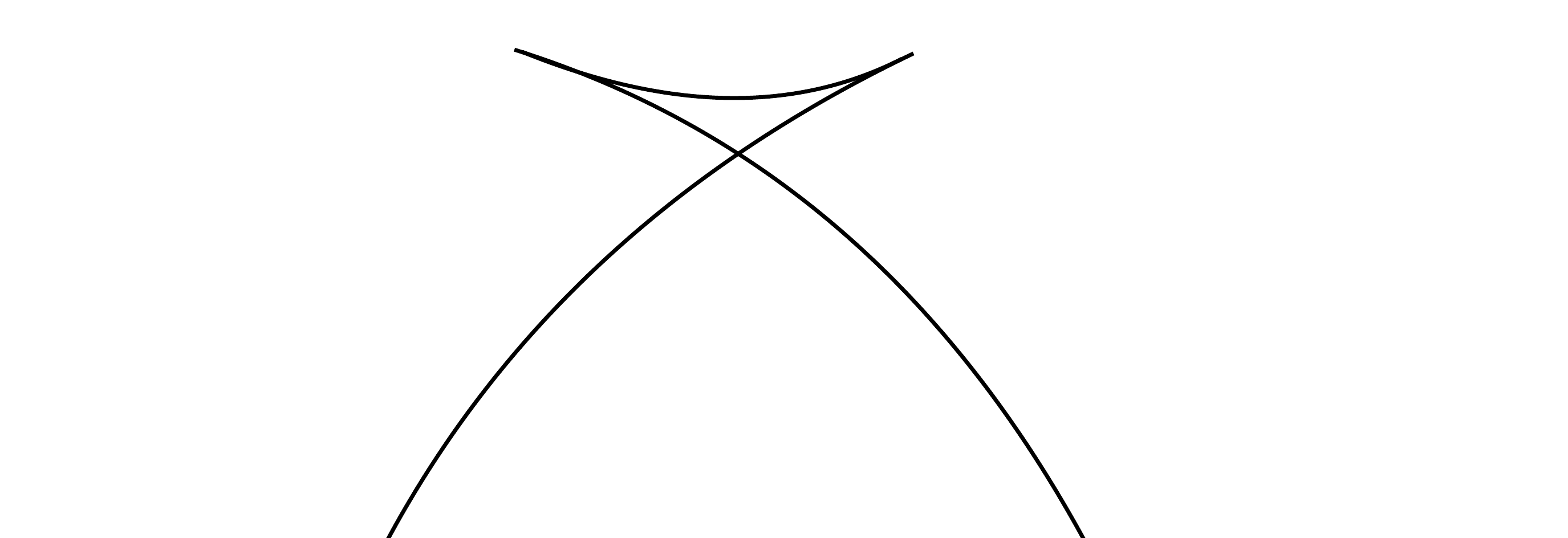}
\caption{The graph of $\tilde{w}^*$ restricted to a vertical line.}
\label{Fig1}
\end{center}
\end{figure}

Second, since $a_{ij} := F_{ij}(D^2\tilde{w})$ is nearly constant in $B_{\kappa}$, we can build a positive super-solution $\varphi$ to the linearized equation $a_{ij}\varphi_{ij} < 0$ in $B_{\kappa}$ that agrees (up to an affine transformation) with $|x|$ to a negative power outside $B_{C\sqrt{\epsilon}}$, is glued to a quadratic polynomial with Hessian eigenvalues smaller than $-10$ in $B_{C\sqrt{\epsilon}}$, and satisfies $|\varphi| \leq C\epsilon$. More precisely, we may assume after an affine transformation that $a_{ij} = \delta_{ij} + O(|x|)$, so for $\kappa$ small, the function $|x|^{-1/2}$ is a super-solution to the linearized equation in $B_{\kappa} \backslash \{0\}$. If we replace this function by a paraboloid in $B_{C\sqrt{\epsilon}}$ with matching values and derivatives on the boundary (call the resulting function $\varphi_0$), then $D^2\varphi_0 = -c\epsilon^{-5/4}I$ in $B_{C\sqrt{\epsilon}}$. Taking $\varphi = C\epsilon^{5/4}\varphi_0$ does the job, since $\varphi_0 \leq C\epsilon^{-1/4}$ in $B_{C\sqrt{\epsilon}}$. We remark that by gluing $|x|^{-1/2}$ to a paraboloid a little more carefully, we may assume that $\varphi$ is smooth.

We conclude for $\epsilon$ small that $\bar{w} := \tilde{w} + \epsilon \varphi$ is a super-solution to the nonlinear equation solved by $\tilde{w}$. Since $D^2\varphi \leq -10I$ in $B_{C\sqrt{\epsilon}}$, the smallest eigenvalue of $D^2\bar{w}$ is everywhere negative, and by similar considerations as in the previous section, $\bar{w}$ has a single-valued Legendre transform. Moreover, $\bar{w}^*$ lies within $C\epsilon^2$ of $\tilde{w}^*$ in $B_{3d/2} \backslash B_{d/4}$ and is a sub-solution of the dual equation solved by $u$ (again by similar reasoning as in the previous section). For the last claim we are using that 
$$F(D^2\bar{w}^*) = \frac{\pi}{2} - F(D^2\bar{w}) > \frac{\pi}{2}-F(D^2\tilde{w}) = -3\theta.$$ 
The maximum principle implies that $u \geq \bar{w}^* - C\epsilon^2$ in $B_{3d/2}$ (indeed, the function on the right lies below $u$ on $\partial B_{3d/2}$ and is a smooth sub-solution to the nonlinear equation solved by $u$). In particular, we have
$$\tilde{w}^* - 2C\epsilon^2 \leq \bar{w}^* - C\epsilon^2 \leq u \leq \text{min}(\tilde{w}^*) = \tilde{w}^*$$
in $B_{3d/2} \backslash B_{d/4}$, establishing the inequality (\ref{C0bound}).

\vspace{3mm}

{\bf Step 3:} Assume by way of contradiction that $\Gamma_u$ is a mass-minimizing integral varifold. By the above considerations, the graphs $\Gamma_u := \{(\nabla u(y),\, y) : y \in B_d\}$ and $\Gamma_{\tilde{w}} = \{(x,\,\nabla \tilde{w}(x)) : x \in B_{\kappa}\}$ are $\epsilon^{2}$-close in $C^1$ when $y$ is restricted to $B_d \backslash B_{d/2}$. Note that $\Gamma_{\tilde{w}}$ is graphical over its tangent $3$-plane $P$ at the origin provided $\kappa$ is small, and lies within a cylinder of radius $C\kappa^2$ around $P$. The same is thus true of $\Gamma_u$ near its boundary, hence everywhere by the maximum principle. There is a competitor for $\Gamma_u$ in the ball ${\bf B}$ of radius $\kappa$ in $\mathbb{R}^6$ obtained by connecting $\Gamma_u$ to $P$ on $\partial {\bf B}$ and replacing $\Gamma_u$ with $P$ in ${\bf B}$ that has mass $(1 + \epsilon_{\kappa})|B_{\kappa}|$, where $\epsilon_{\kappa} << 1$ for $\kappa$ small (we in fact have $\epsilon_{\kappa} \leq C\kappa^3$). Since this bounds the mass of $\Gamma_u$ from above in ${\bf B}$, for $\kappa$ small we can apply the Allard theorem (as stated e.g. in Theorem 3.2 and the remark thereafter in \cite{DL}, see also \cite{A} and \cite{Si}) to conclude that $\Gamma_u$ is smooth in ${\bf B}/2$. Moreover, $\Gamma_u$ and $\Gamma_{\tilde{w}}$ are the graphs over $P$ of maps that are $\epsilon^{2}$-close in $C^1$ in an annulus, have gradient bounded by $C\kappa$ (for $\Gamma_{\tilde{w}}$ this is just from Taylor expansion and for $\Gamma_u$ this follows from a quantitative form of the Allard theorem, see e.g. Theorem 2.1 in \cite{DGS}) and solve the minimal surface system (this last statement due to the fact that $\Gamma_u$ and $\Gamma_{\tilde{w}}$ are mass-minimizing, so in particular they are critical points of area). Using the standard regularity theory of systems arising as Euler-Lagrange equations of functionals with uniformly convex integrands (the area element is uniformly convex for maps with small gradient) we infer that $\Gamma_u$ and $\Gamma_{\tilde{w}}$ are everywhere $\epsilon^{2}$-close in the $C^1$ sense. In particular, $\Gamma_u$ is graphical in the $x$ variable as well, so $u^*$ is single-valued and satisfies
$$D^2u^*(0) = D^2\tilde{w}(0) + O(\epsilon^{2}) > C^{-1}\epsilon I$$
for $\epsilon$ small. This implies that $D^2u(\nabla u^*(0))$ is a positive matrix, contradicting the equation for $u$.

\vspace{3mm}

One feature of the examples in this section is that the singularities occur near the center of a ball, in contrast with the examples in the previous section, which are only constructed in a small neighborhood of a singularity. Another feature is that the singularities of the examples in this section exist for all choices of $\epsilon$ small, illustrating their stable nature. 

The argument above shows that $u$ is not well-approximated by $\tilde w^*$ near the origin. Instead, we need to consider the Legendre transform of a solution $v$ to the modified equation 
$$ \max \{ F(D^2 v)- \frac \pi 2 - 3 \theta, \det \, D^2 v\} =0, \quad \quad v= \tilde w \quad \mbox{on $\partial B_\kappa$.}$$
This is in fact the starting point of our construction in Theorem \ref{Main}, in which we exhibit a $C^{2,1}$ solution of \eqref{Bellman} with an analytic compact free boundary between the two operators. 

We expect that the examples constructed in the previous section are in fact local models for the singularities appearing in this section. More precisely, we conjecture that for all $\epsilon > 0$ small, the examples $u$ constructed in this section exhibit Lipschitz singularities, and moreover that their Legendre transforms $u^*$ solve degenerate Bellman equations with compact free boundaries. The main difficulty consists in showing that solutions $v$ to the equation above are of class $C^2$ and have injective gradient.
On the other hand, after an appropriate rescaling, as $\epsilon \to 0$ the equation linearizes  to a model equation of the type
$$ \max \{ \triangle v, v_{33} + 1 - |x|^2\} =0, \quad \quad v \to 0 \quad \mbox{as $|x| \to \infty$.}$$
This problem has a compact free boundary which seems to have good regularity properties. We intend to analyze these questions further in a subsequent work. 



\end{document}